\documentclass{amsart}
\usepackage{amsmath,amssymb}
\usepackage{graphicx}
\newtheorem{theorem}{Theorem}[section]
\newtheorem{proposition}[theorem]{Proposition}
\newtheorem{corollary}[theorem]{Corollary}

\newtheorem{lemma}[theorem]{Lemma}
\newtheorem*{principle}{Heuristic principle}

\theoremstyle{definition}
\newtheorem{definition}[theorem]{Definition}
\newtheorem{assumption}{Assumption}
\newtheorem{example}[theorem]{Example}

\theoremstyle{remark}
\newtheorem{remark}[theorem]{Remark}

\numberwithin{equation}{section}

\newcommand{\al}{\alpha}

\newcommand{\de}{\delta}
\newcommand{\ep}{\epsilon}

\newcommand{\ga}{\gamma}

\newcommand{\la}{\lambda}
\newcommand{\om}{\omega}

\newcommand{\te}{\theta}

\newcommand{\De}{\Delta}
\newcommand{\Ga}{\Gamma}
\newcommand{\La}{\Lambda}
\newcommand{\Si}{\Sigma}
\newcommand{\Om}{\Omega}
\newcommand{\hG}{\widehat{G}}
\newcommand{\hF}{\widehat{\mathcal F}}

\newcommand{\hD}{\widehat{D}}

\def\CC{\mathbb{C}}
\def\NN{\mathbb{N}}
\def\RR{\mathbb{R}}

\def\ZZ{\mathbb{Z}}

\renewcommand\SS{\mathbb{S}}
\newcommand\BG{{\overline G}}
\newcommand\By{{\bar y}}

\newcommand{\cC}{{\mathcal C}}

\newcommand{\cF}{{\mathcal F}}
\newcommand{\cG}{{\mathcal G}}

\newcommand{\cL}{{\mathcal L}}
\newcommand{\cM}{{\mathcal M}}

\newcommand{\cP}{{\mathcal P}}

\newcommand\cZ{\mathcal Z}
\newcommand{\pd}{\partial}
\newcommand\minus\backslash

\newcommand{\ms}{\mspace{1mu}}
\newcommand\lan\langle
\newcommand\ran\rangle

\newcommand{\I}{{\mathrm i}}
\newcommand{\e}{{\mathrm e}}

\DeclareMathOperator\Real{Re}

\DeclareMathOperator\diag{diag} 
\DeclareMathOperator\dist{dist}

\newcommand\DD{\mathbb D}
\renewcommand\leq\leqslant
\renewcommand\geq\geqslant
\newlength{\intwidth}

 \DeclareMathOperator\ind{ind}

\begin{document}

\title[Gradient dynamical systems
  on surfaces and Green's functions]{Gradient dynamical systems
  on open surfaces and critical points of Green's functions}

\author{Alberto Enciso}
\address{Instituto de Ciencias Matem\'aticas, Consejo Superior de
  Investigaciones Cient\'\i ficas, 28049 Madrid, Spain}
\email{aenciso@icmat.es}

\author{Daniel Peralta-Salas}
\address{Instituto de Ciencias Matem\'aticas, Consejo Superior de
  Investigaciones Cient\'\i ficas, 28049 Madrid, Spain}
\email{dperalta@icmat.es}

%
%
\begin{abstract}
  We study the dynamics of the vector field on an open surface given
  by the gradient of a Green's function. This dynamical approach
  enables us to show that this field induces an invariant
  decomposition of the surface as the union of a disk and a
  $1$-skeleton that encodes the topology of the surface. We analyze
  the structure of this $1$-skeleton, thereby obtaining, in particular, a
  topological upper bound for the number of critical points a Green's
  function can have. Connections between the dynamical properties of
  the gradient field and the conformal structure of the surface are
  also discussed.
\end{abstract}
\maketitle

\section{Introduction}
\label{S:intro}

Let $M$ be a noncompact surface without boundary of class $C^\infty$,
endowed with a smooth
complete Riemannian metric $g$. We denote
by 
\[
\cG: (M\times M)\minus\diag(M\times M)\to\RR
\]
a Green's function
of $M$, which is defined as a symmetric function (i.e., $\cG(x,y)=\cG(y,x)$) that
satisfies the equation
\begin{equation}\label{Gr}
\De_g\cG(\cdot,y)=-\de_y
\end{equation}
for each $y\in M$. That is to say, if one considers the action of the
Laplace--Beltrami operator of the manifold,
$\De_g$, on the Green's function $\cG(x,y)$ (with respect to the first
variable~$x$), one obtains a Dirac measure supported at the point $y$.

Our goal in this paper is to analyze the dynamical properties of the
gradient of the Green's function. For this, we will find it
notationally convenient to fix a point $y\in M$, once and for all, and
use the notation $G:=\cG(\cdot,y)$ for the Green's function with pole
$y$, which is smooth and harmonic in $M\minus\{y\}$. Therefore, the
gradient field we will study in this paper will be $\nabla_g G$.

The study of Green's functions is a central topic in Riemannian
geometry and geometric analysis. Hence, there is a vast related
literature covering, among many other aspects, the existence of
positive Green's functions~\cite{CY75,LT87}, upper and lower bounds,
gradient estimates and asymptotics~\cite{LY86,CM97}, and the
connection between Green's functions and the heat
kernel~\cite{LTW97,GS02}.

What is somewhat surprising is that the dynamical properties of the
gradient field $\nabla_gG$ remain virtually unexplored, with the
exception of the classical work of Brelot and Choquet~\cite{BC51}. Of
course, a key issue in the study of this vector field is the analysis
of the critical set of the Green's function. This question is of
considerable interest by itself, and deeply related with other
problems, set in significantly easier contexts, that date back to the
1950s (see e.g.~\cite{Wa50,CM69,Mo70,Sh80} and references
therein). Indeed, some of these articles were motivated in part by the
fact that in Euclidean space the Green's function arises as the
electric potential of a charged particle, so that its critical points
correspond to equilibria and the trajectories of its gradient field
are the force lines studied by Faraday and Maxwell in the XIX century
(nontrivial contributions to this problem were made in the recent
paper~\cite{Eremenko}). As a side remark, let us point out that
another elliptic PDE (very different from~\eqref{Gr}) in which the
analysis of the critical points of its solutions has recently
attracted considerable attention can be found in~\cite{Lin}.

One reason why the study of the dynamical properties of the gradient
field $\nabla_gG$ (and in particular of the critical set of $G$) is
hard to tackle, from the point of view of geometric analysis, is that
the estimates for Green's functions obtained through PDE methods are
not sufficiently fine to elucidate whether the gradient of $G$
vanishes in a certain region. Moreover, the noncompactness of the
underlying surface introduces complications related to the behavior of
the Green's function at infinity. 

In this paper we will show how these difficulties can be overcome by
exploiting the conformal properties of the surface and resorting to
techniques of gradient dynamical systems. Our approach will lead to a
topological upper bound for the number of critical points of the
Green's function and a complete description of the local and global
dynamics of the gradient field $\nabla_gG$.

To some extent, the core of this paper is the remarkable heuristic
principle we will now state, which links the dynamics of the
gradient field $\nabla_gG$, defined in terms of solutions to an
elliptic PDE, with the topology of the underlying surface. It must be
stressed that this principle will be promoted to a rigorous statement
(after introducing the necessary tools and notation) in
Corollary~\ref{C.cF} and Theorem~\ref{P.hF2}:

\begin{principle}\label{principle}
  Suppose that the surface $M$ is of finite type. Then $M$ can
  be decomposed as the union of a disk~$D$ and a (possibly
  disconnected) noncompact graph $\cF$, both of which are invariant
  under the local flow of the gradient field $\nabla_gG$. The disk
  consists of the pole $y$ and the points of $M$ whose $\om$-limit is
  $y$. The graph $\cF$ consists of the critical points of $G$, their
  stable components, and certain trajectories of $\nabla_gG$ that
  escape to infinity. When suitably compactified, $\cF$ is a connected
  graph that encodes the topology of the surface, the rank of the
  first homology group of $\cF$ being twice the genus of $M$.
\end{principle}

The characterization of the set $\cF$ that will emerge from the
rigorous version of this heuristic principle yields, as a nontrivial
application, the  following topological upper bound for the number of critical
points of the Green's function:

\begin{theorem}\label{T:main}
  Suppose that the surface $M$ is of finite type, that is, its
  fundamental group has finite rank. Then the number of critical
  points of any Green's function $G$ on $M$ is not larger than twice
  the genus of $M$, $\nu$, plus the number of ends, $\la$, minus~$1$:
\[
\# \ms\text{critical points}\leq 2\nu+\la-1\,.
\]
If this upper bound is attained then $G$ is a Morse function.
\end{theorem}

In fact, the analysis of the set $\cF$ does not only yield this
topological upper bound, but a more refined bound that exploits the
conformal structure of the surface (see Theorem~\ref{T.main2}). This
is particularly interesting because, as we shall see, it establishes
some links between the conformal geometry of the surface and the
portrait of the gradient field $\nabla_gG$. It should be stressed that
an analogous result does not hold for higher-dimensional Riemannian
manifolds, as shown in~\cite{EP12}.

A different proof of the estimates for the number of critical points
of $G$, relying on methods from elliptic PDEs, was recently given by
the authors in~\cite{EP12}. However, the dynamical approach taken in the present
paper provides a very satisfactory picture of the invariant sets
connecting the different critical points of $G$ and the dynamics of
the field $\nabla_gG$, which cannot be obtained with the PDE
techniques used in~\cite{EP12}.

The article is organized as follows. In Section~\ref{S:LiTam} we will
present some basic facts regarding Green's functions on surfaces,
including their obtention through an exhaustion procedure, their
behavior at infinity and their connection with the conformal structure
of the surface. In Section~\ref{S:Local} we describe the dynamics of
the field $\nabla_gG$ in a neighborhood of the pole $y$ or a critical
point. In Section~\ref{S:basin} we introduce a convenient
compactification of the surface and establish some key properties of
the sets $D$ and $\cF$ introduced in the Heuristic Principle above
(and of some compactifications thereof). The structure of the set
$\cF$ and its compactification is characterized in
Section~\ref{S:bounds}, which allows us to prove the upper bound for
the number of critical points of the Green's function. To conclude, in
Section~\ref{S:remarks} we discuss the connection between the dynamics
of the field $\nabla_gG$ and the conformal geometry of the underlying
surface and make some comments regarding surfaces of infinite
topological type.

\section{Green's functions on surfaces}
\label{S:LiTam}

In this section we will recall what a Green's function is and how to
obtain them using an exhaustion procedure, placing a special emphasis
on conditions ensuring that the Green's function is ``well behaved''
at infinity, which is a key ingredient in the analysis of the dynamical
properties of the vector field $\nabla_g G$. We will also discuss how
to exploit conformal isometries to classify the possible behavior of
the Green's function at the ends of the surface. Throughout this
paper, the surfaces will be of class $C^\infty$ and of finite
topological type (that is, with finitely generated fundamental group)
unless stated otherwise.

The reason why it is crucial to make assumptions on the behavior of
the Green's function at infinity can be readily seen even in the
simplest case: the Euclidean plane $\RR^2$. Indeed, if we let $h$ be
any harmonic function, it is clear that any function of the
form
\[
\cG(x,y)=-\frac1{2\pi}\log|x-y|+h(x)+h(y)
\]
is symmetric and satisfies the Green's function equation~\eqref{Gr}. 
The standard way of deciding which of these Green's functions
should be ``admissible'' is to demand that the Green's
function be obtained as a limit of Dirichlet Green's functions
associated with an exhaustion of the plane by compact subsets (more
details on this point will be given below). The only Green's function
arising from such an exhaustion procedure would be the usual one,
\begin{equation}\label{eqlog}
\cG(x,y)=-\frac1{2\pi}\log|x-y|\,,
\end{equation}
which is the relevant Green's function for all geometric or
analytical considerations. 

In a general Riemannian surface $(M,g)$, we will assume that the
Green's function $G$ we consider shares the following two properties
with the above Green's function~\eqref{eqlog} of the Euclidean
plane. The first assumption is a weak monotonicity property for the
Green's function on circles. The second assumption says that, when the
surface admits a {\em minimal (positive) Green's function}\/ (as is
the case of the hyperbolic plane, although not of the Euclidean one),
we should always consider this Green's function, for it plays a very
special role in analysis and geometry. Recall that a positive Green's
function is minimal when it is pointwise smaller than any other
positive Green's function. Observe that, if the surface
$(M,g)$ does not admit a minimal Green's function, then the infimum of
any Green's function $G$ over the surface is $-\infty$.

\begin{assumption}[Monotonicity]\label{P1}
The Green's function $G$ is nondecreasing in the sense that
\begin{equation}\label{decreasing}
\sup_{M\minus B_g(y,r)} G=\max_{\pd B_g(y,r)}G\qquad\text{for all }r>0\,,
\end{equation}
where $B_g(y,r)$ denotes the geodesic disk in $M$ centered at the pole $y$
of radius $r$.
\end{assumption}

\begin{assumption}[Minimality]\label{P2}
$G$ is the minimal Green's function whenever the Riemannian surface admits a positive Green's function.
\end{assumption}

It should be stressed that on any Riemannian surface $(M,g)$ there are
Green's functions that satisfy these assumptions. Indeed, the way one
shows there always exist Green's functions on any surface is by taking
a suitable limit of the Dirichlet Green's functions of an exhaustion
of the surface by bounded domains. Since the Green's functions
one obtains in this fashion actually satisfy the above
assumptions, {\em throughout this paper we will assume that
  the Green's function satisfies Assumptions~\ref{P1} and~\ref{P2}.}
For the benefit of the reader, we will next review this construction
of Green's functions through an exhaustion by compact sets, which was introduced in the
context of general Riemannian manifolds by Li and Tam~\cite{LT87}.  To
simplify the exposition, we will consider the case of Green's
functions $G(x)$ with a fixed pole $y$, but actually the procedure
automatically yields the symmetric function $\cG(x,y)$.

Consider an exhaustion $\Om_1\subset\Om_2\subset\cdots$ of the surface
$M$ by bounded domains. We can assume without loss of generality that
the pole $y$ belongs to the first domain~$\Om_1$. The idea now is to
impose Dirichlet boundary conditions on each bounded domain $\Om_j$ and consider the
corresponding Green's function $G_j:\Om_j\minus\{y\}\to\RR$, which
satisfies the equation
\[
\De_g G_j=-\de_y\quad\text{in }\Om_j\,,\qquad G_j=0\quad \text{on
}\pd\Om_j\,.
\]
One would be tempted to define the Green's function $G$ as the limit
of $G_j$ as $j\to\infty$. However, this limit does not exist in
general. What can be proved (see e.g.~\cite{LT87}) is that, for any
choice of the domains $\Om_j$, one can take a sequence of nonnegative real numbers
$(a_j)_{j=1}^\infty$ such that $G_j-a_j$ converges uniformly on
compact sets of $M\minus\{y\}$ to a Green's function $G$ with pole
$y$. Generally, the Green's functions obtained through this procedure are
non-unique, but they satisfy the monotonicity
assumption~\eqref{decreasing} and, when the surface admits a positive Green's function,
this construction always yields the minimal one.

Incidentally, it is worth pointing out that Green's functions do not
exist on closed surfaces, which is the reason why we just consider
noncompact surfaces in this paper. To see why, it is enough to suppose
that there is a solution of Eq.~\eqref{Gr} in a closed surface $M$
and, with a slight abuse of notation,
integrate both sides of the equation over the whole surface and
integrate by parts, which would yield the contradiction
\[
0=\int_M\De_g G\,=-\int_M \de_p=-1\,.
\]

Let us now pass to describe how one can utilize conformal isometries
to analyze the behavior of the Green's function at each end. Recall
that two Riemannian surfaces $(M,g)$ and $(\cM,g_0)$ are {\em
  conformally isometric} if there is a diffeomorphism $\Phi:\cM\to M$
and a smooth positive function $f$ on $\cM$ such that $\Phi^* g= f
g_0$. An important property of the Laplace equation on surfaces is its
{\em conformal invariance}, that is, if $G(x)$ satisfies the equation
\[
\De_gG=-\de_y
\]
on the surface~$M$ for some point $y$, and another
surface $\cM$ is
conformally isometric to $M$ through a diffeomorphism $\Phi:\cM\to M$,
then 
\[
\BG(x):=G(\Phi(x))
\]
is a Green's function of $\cM$ with pole $\By:= \Phi^{-1}(y)$:
\begin{equation}\label{eqBG}
\De_{g_0}\BG=-\de_{\By}\,.
\end{equation}
Furthermore, the corresponding gradient fields are orbitally conjugated through
the relation
\begin{equation}\label{conjug}
\nabla_{g_0}\BG=f\,\Phi^*(\nabla_gG)\,.
\end{equation}

A key result in the conformal geometry of surfaces, which will be of great
use in this paper, is the uniformization theorem:

\begin{theorem}[Uniformization]\label{T.unif}
There is a compact surface $\Si$ of genus $\nu$ with a metric of
constant curvature $g_0$, a certain number $\la_1\geq0$ of isolated points $p_i\in\Si$ and
another number $\la_2\geq0$ of closed disks $D_i\subset\Si$ with smooth
boundary such that the Riemannian surface $M$ is
conformally isometric to $(\cM,g_0)$, with
\begin{equation}\label{cM}
\cM:=\Si\ms\minus\bigg(\bigcup_{i=1}^{\la_1}\{p_i\}\cup
\bigcup_{j=1}^{\la_2} D_j\bigg)\,.
\end{equation}
\end{theorem}
As is customary, we will call the deleted points $\{p_1,\dots,
p_{\la_1} \}$ the {\em parabolic ends}\/ of the surface $M$, while the
deleted disks $\{D_1,\dots, D_{\la_2}\}$ are its {\em hyperbolic
  ends}\/. Furthermore, we will say that the parabolic end $p_i$ is a {\em
  removable singularity}\/ if the function $\BG$ can be extended so as
to satisfy the equation
\[
\De_{g_0}\BG=0
\]
in a neighborhood of $p_i$ in $\Si$. It should be noticed that an end
being parabolic or hyperbolic is a geometric property of the surface,
related to its conformal structure. However, whether a parabolic end
is removable or not is not a geometric issue, as it depends not only
on the surface $M$ but also on the Green's function we consider.

In the following two propositions we will relate the existence of
parabolic and hyperbolic ends with the behavior of the 
function $\BG$ at each deleted point or disk:

\begin{proposition}\label{P.parends}
If all the ends of the surface $(M,g)$ are parabolic, the surface does
not admit any positive Green's functions. Moreover, at each end $p_i$
we have that either
\begin{equation}\label{minf}
\lim_{x\to p_i} \BG(x)=-\infty
\end{equation}
or $p_i$ is a removable singularity. There is at least one point $p_i$ where
the condition~\eqref{minf} is satisfied.
\end{proposition}

\begin{proof}
When the number of hyperbolic ends $\la_2$ is $0$, it follows from
Eq.~\eqref{eqBG} that the function $\BG$ satisfies the equation
$\De_{g_0}\BG=0$ everywhere in $\Si$ but at the pole $\By$ and the isolated points
$p_i$. Furthermore, by Assumption~\ref{P1}, $\BG$ is upper
bounded at each point $p_i$. If it is also lower bounded, it is
standard that $p_i$ is a removable singularity~\cite{GS55} and $\De_{g_0}\BG=0$ in
a neighborhood of $p_i$. If $\BG$ is not lower bounded, $p_i$ is an
isolated singularity of $\BG$, and the fact that $\BG$ is upper
bounded readily implies that $\De_{g_0}\BG=c_i\de_{p_i}$ in a neighborhood of
$p_i$ for some non-negative constant~$c_i$. Hence
\begin{equation}\label{juntardes}
\De_{g_0}\BG=-\de_{\By}+\sum_{i=1}^{\la_1}c_i\,\de_{p_i}
\end{equation}
in the closed manifold $\Si$, so integrating this equation over $\Si$
and using that $\int_\Si \De_{g_0}\BG=0$ we infer that $\sum_i
c_i=1$. Hence $c_i$ is positive for some $i$ and, in view of the
asymptotic behavior of any Green's function at a pole, it stems that
$\BG$ tends to $-\infty$ at the corresponding point $p_i$.
\end{proof}

\begin{proposition}\label{P.hypends}
If the surface $(M,g)$ has at least one hyperbolic end, there is a
minimal positive Green's function $G$. The corresponding function $\BG$ tends
to zero at the boundary of each disk $D_i$ and all the parabolic ends
$p_i$ are removable singularities.
\end{proposition}

\begin{proof}
  When the number of hyperbolic ends is $\la_2\geq 1$, it is standard
  (for example, due to the existence of nonconstant positive harmonic
  functions in $\cM$~\cite{LT87}) that the surface
  admits a positive Green's function. Therefore,
  Assumption~\ref{P2} ensures that $G$ is the minimal Green's function
  of the surface, which corresponds to the unique solution $\BG$ of the
  boundary problem
\[
\De_{g_0}\BG=-\de_{\By}\quad \text{in } \Si\ms\minus
\bigcup_{j=1}^{\la_2} {D_j}\,,\qquad \BG=0\quad\text{on }\pd D_j\text{
  for all }j.
\]
In particular, if there are any parabolic ends, they are all removable singularities.
\end{proof}

Because of these propositions, {one can extend the function $\BG$ and
  the gradient field $\nabla_{g_0}\BG$ to all the removable
  singularities of the surface.} We will find it occasionally
convenient to consider this extension, which we will not distinguish
notationally (it will be clear from the context).  To conclude this section, we will present two examples
that illustrate the issue of uniqueness and non-uniqueness of Green's
functions on surfaces.

\begin{example}
  Consider the plane $\RR^2$ with its Euclidean metric. It has a
  parabolic end, so that it is conformally isometric to the round sphere
  $(\SS^2,g_0)$ minus a point $p$ via a diffeomorphism
  $\Phi:\SS^2\minus\{p\}\to\RR^2$. 

  The standard Green's function with pole $y$,
\begin{equation}\label{GR2}
G(x):=-\frac1{2\pi}\log|x-y|\,,
\end{equation}
obviously satisfies Assumption~\ref{P1}. Moreover, it is the only
Green's function satisfying this assumption, even though the plane
does not admit any positive Green's functions. To see this, take the
function $\BG$ corresponding to any Green's function satisfying
Assumption~\ref{P1}. Since the assumption is satisfied,
Eq.~\eqref{juntardes} in the proof of Proposition~\ref{P.parends}
shows that $\BG$ must satisfy the equation
\[
\De_{g_0}\BG=-\de_{\By}+\de_p
\]
in the whole $\SS^2$. Hence $\BG$ is uniquely determined, so that the
Green's function must be given by~\eqref{GR2}.
\end{example}

\begin{example}
Let us consider the flat cylinder $\RR\times\SS^1$ and natural
coordinates $(z,\te)$. It is conformally equivalent to the round sphere $(\SS^2,g_0)$
minus two points $\{p_1,p_2\}$, so it does not admit a positive
Green's function. 

A Green's function with pole at
$(z_0,\te_0)$ is
\[
G_1(z,\te):=-\frac1{4\pi}\log\big[ \cosh (z-z_0)-\cos(\te-\te_0)\big]\,.
\]
This Green's function satisfies Assumption~\ref{P1}. It is not the
only Green's function on the flat cylinder with this property; e.g.,
one can check that
\[
G_2(z,\te):= -\frac1{4\pi}\log\big[\tfrac12\e^{2z}+ \tfrac12\e^{2z_0}-\e^{z+z_0}\cos(\te-\te_0)\big]
\]
is another instance. These Green's functions are connected by the
identity
\[
G_2(z,\te)=G_1(z,\te)-\frac{z+z_0}{4\pi}\,.
\]
Notice that $G_1$ tends to $-\infty$ at both ends (that is, as
$z\to\pm\infty$) while $G_2$ tends to~$-\infty$ as $z\to\infty$ but
the end $z\to-\infty$ corresponds to a removable singularity of~$G_2$.
\end{example}

\section{Local dynamical properties of Green's functions}
\label{S:Local}

In this section we will carry out a local study of the dynamics of the
gradient of the Green's function in a neighborhood of the pole or a
critical point. Since the fields $\nabla_gG$ and $\nabla_{g_0}\BG$ are
orbitally conjugated (cf.\ Eq.~\eqref{conjug}),
for convenience we will work with the latter gradient field instead.

The first proposition we will prove here asserts that, when multiplied by a suitable
factor, the gradient vector field $\nabla_{g_0}\BG$ can be smoothly
linearized at the point $\By$, and that the corresponding normal form is
a stable node. In particular, the trajectories approach the pole
with a well-defined tangent.

\begin{proposition}\label{L.pole}
There are $C^\infty$ coordinates $(x_1,x_2)$, defined in a
neighborhood $U$ of the point $\By$ in $\cM$ and centered at $\By$, and a smooth nonnegative
function $\rho:U\to\RR$ that only vanishes at the pole, such that the
gradient field $\nabla_{g_0} \BG$ can be written as
\[
\rho\,\nabla_{g_0} \BG=- x_1\,\frac\pd{\pd x_1}-x_2\,\frac\pd{\pd x_2}
\]
in the domain $U$.
\end{proposition}
\begin{proof}
Let us take local isothermal coordinates $u=(u_1,u_2)$ centered at $\By$,
in which the metric reads as
\[
g_0=f(u)\,\big( du_1^2+du_2^2\big)
\]
for a positive function $f$. Therefore, Eq.~\eqref{eqBG} can be written in these coordinates~as
\[
\frac{\pd^2\BG}{\pd u_1^2}+ \frac{\pd^2\BG}{\pd u_2^2}=-\de_0\,,
\]
so $\BG$ must be of the form
\[
\BG=-\frac1{2\pi}\log|u|+ h(u)\,,
\]
with $|u|^2:=u_1^2+u_2^2$ and $h$ a harmonic function:
\[
\frac{\pd^2h}{\pd u_1^2}+ \frac{\pd^2h}{\pd u_2^2}=0\,.
\]
The gradient of $\BG$ can then be expressed as
\[
\rho\, \nabla_{g_0}
\BG=-\bigg(u_1-2\pi|u|^2\frac{\pd h}{\pd u_1}\bigg)\,\frac\pd{\pd
  u_1}-\bigg(u_2-2\pi|u|^2\frac{\pd h}{\pd u_2}\bigg)\,\frac\pd{\pd u_2}\,,
\]
with $\rho(u):=2\pi f(u)\, |u|^2$. The origin is then a hyperbolic zero
of the vector field $\rho\, \nabla_{g_0}\BG$ (which can be smoothly extended
to the origin) and the corresponding eigenvalues are $(-1,-1)$. Hence Siegel's theorem
ensures that $\rho\, \nabla_{g_0}\BG$ can be linearized via a diffeomorphism that is an
analytic function of the coordinates $(u_1,u_2)$ and the claim follows.
\end{proof}

In the following proposition we will characterize the structure of the
trajectories of the field $\nabla_{g_0}\BG$ in a neighborhood of a
critical point of $\BG$. In particular, we see that the trajectories
approaching the critical point have a well-defined tangent. This
proposition can be seen as a dynamical analog of Cheng's result on the
critical points of eigenfunctions on surfaces~\cite{Ch76}. Let us
recall that the {\em stable (resp.\ unstable) set}\/ of a zero $z$ of
a vector field is given by the points whose $\om$-limit (resp.\
$\al$-limit) is exactly the point $z$.

\begin{proposition}\label{L:halfb}
  Let $z$ be a zero of the vector field $\nabla_{g_0}\BG$ (possibly a
  removable singularity $p_i$) and let $m\geq2$ be the degree of the
  lowest nonzero homogeneous term in the Taylor expansion of
  $\BG-\BG(z)$ at $z$ (which is always finite). Then $z$ is an isolated zero and the
  intersection of a small neighborhood of $z$ in $\Si$ with either its
  stable or unstable set is homeomorphic to the set
\[
\big\{\zeta\in\CC:\zeta^m\in[0,1) \big\}\,,
\]
In particular, the index of the point $z$ is $1-m$.
\end{proposition}

\begin{proof}
Let us take isothermal coordinates $u=(u_1,u_2)$ centered at $z$, in
which the metric reads as
\[
g_0=f(u)\,\big(du_1^2+du_2^2\big)\,.
\]
Therefore, $\BG$ is a harmonic function of $u$ with respect to
the flat metric:
\begin{equation}\label{harm}
\frac{\pd^2\BG}{\pd u_1^2}+ \frac{\pd^2\BG}{\pd u_2^2}=0\,.
\end{equation}
It is therefore standard that $\BG$ is an analytic function of
$(u_1,u_2)$.

Let $h_m$ be the first nonzero homogenous polynomial that appears in
the Taylor expansion of $\BG$ in these coordinates at $0$, so that
(with a slight abuse of notation)
\begin{subequations}\label{new}
\begin{align}
\BG(u)-\BG(0)&=h_m(u)+O(|u|^{m+1})\,,\\
f(u)\, \nabla_{g_0}\BG(u)&=\nabla h_m(u)+O(|u|^m)\,,\label{eq2b}
\end{align}
\end{subequations}
Here we are denoting by $\nabla$ the flat space gradient with respect
to the coordinates~$u$.

By Eq.~\eqref{harm}, the homogeneous polynomial $h_m$ is a harmonic
function (with respect to the flat space Laplacian in the coordinates
$u$), so it must be of the form
\[
h_m=C\Real \big[\e^{\I \al}(u_1+\I u_2)^m\big]
\]
for some real constants $C,\al$. In particular, $0$ is an isolated
critical point of $h_m$, which readily implies that $z$ is an isolated
zero of $\nabla_{g_0}\BG$.

Let us now consider polar coordinates
$(r,\theta)\in(0,\ep)\times\SS^1$ defined by
$(u_1,u_2)=(r\,\cos\theta,r\,\sin\theta)$. In these coordinates one
has
\[
h_m(r,\theta)=C\,r^m\cos (m\ms \theta+\al)\,.
\]
There is no loss of generality in
setting $\al=0$.
We define the polar blow up of the gradient $\nabla_{g_0}\BG$
at $z$ using polar coordinates as the vector field
\[
X:=\frac{f}{Cm\,r^{m-2}}\,\nabla_{g_0}\BG= \frac1{Cm\,r^{m-2}}\,\big(\nabla
h_m+O(r^m)\big)\,,
\]
where we have used Eq.~\eqref{eq2b}. The blown-up trajectories are then given by
\begin{subequations}\label{dotrte}
\begin{align}
\dot r&=r\,\cos m\theta+O(r^2)\,,\\
\dot\theta&=-\sin m\theta+O(r)\,.
\end{align}
\end{subequations}
The blown-up critical points are thus $(0,\theta_k)$,
with $\theta_k:=k\pi/m$ and $k=1,\dots, 2m$. The Jacobian matrix
of $X$ at $(0,\theta_k)$ is
\begin{equation}\label{Jac}
DX(0,\theta_k)=\left(%
\begin{array}{cc}
   (-1)^k& 0 \\
  0 & (-1)^{k+1} \\
\end{array}%
\right)\,,
\end{equation}
so these critical points are hyperbolic saddles. By blowing down, we
immediately find that a deleted neighborhood of $0$ consists exactly
of $2m$ hyperbolic sectors of the vector field $X$. 

Since the field $X$ is proportional to the gradient field $\nabla_{g_0}\BG$
through a factor that does not vanish but at $z$, this shows that
the intersection with a small neighborhood of $z$ with the its stable
or unstable set is homeomorphic to
\[
\big\{\zeta\in\CC:\zeta^m\in[0,1) \big\}\,,
\]
as claimed. Besides, the well known Bendixson formula for the index of a
planar vector field asserts
that the index of $z$ is 
\[
\ind(z)=1-\frac{\text{number of hyperbolic sectors}}2=1-m\,,
\]
as claimed.
\end{proof}

It is worth mentioning that the dynamics of the gradient of a harmonic
function in a neighborhood of a critical point in dimension higher
than $2$ is much more involved, as is discussed in~\cite{Go09}.

\section{The basin of attraction of the pole}
\label{S:basin}

In this section we will introduce the concept of basin of attraction
associated with $\nabla_{g_0}\BG$, as well as some
convenient compactifications thereof. We shall see how this object and
its boundary define a decomposition of the surface as the union of a
disk and a 1-skeleton, as outlined in the Heuristic Principle in the
Introduction. We shall also establish some properties
of these sets.

The basin of attraction of the pole $y$ is a key object in this paper,
and can be thought of as the set of points of the surface $M$ that approach $y$ when
flowed along the trajectories of the field $\nabla_gG$. However, in view of
the characterization of $M$ in terms of a compact surface $\Si$ (the
Uniformization Theorem~\ref{T.unif}), it is
slightly more convenient to study the basin of attraction directly in
this compact surface, so instead we will use the following

\begin{definition}
The {\em basin of attraction}\/ is the set of points $D$ in $\cM$ whose $\om$-limit
along the trajectories of $\nabla_{g_0}\BG$ is $\By$:
\begin{equation}\label{eqD}
D:=\big\{x\in \cM: \om(x)=\By\big\}\,.
\end{equation}
\end{definition}
Of course, by the relationship between $\nabla_{g_0}\BG$ and
$\nabla_gG$, the diffeomorphism $\Phi:\cM\to M$ maps the basin $D$
onto the set of points in $M$ whose $\om$-limit along the integral
curves of $\nabla_gG$ is the pole $y$. It is easy to prove that $D$ is
diffeomorphic to a disk.

It is clear that both $D$ and its complement in $\cM$ are invariant sets under
the local flow of $\nabla_{g_0}\BG$. The complement,
\[
\cF:=\cM\minus D\,,
\]
will be a crucial object in the rest of the paper. 

In the following proposition we shall prove a general result about sets of $\cM$ that are
invariant under the local flow of $\nabla_{g_0}\BG$ from which it stems an
important property of $\cF$ as a corollary:
that $\cM\minus D$ has empty interior and thus $\cF$ coincides with
the boundary of the basin $D$ in $\cM$. Notice that, since the basin
of attraction is contractible, $\cF$ cannot be empty
unless the surface $M$ is diffeomorphic to $\RR^2$.


\begin{proposition}\label{P.interior}
  Let $S\subset \cM$ be an invariant set under the flow of $\nabla_{g_0}
  \BG$ that does not contain the point $\By$ and is relatively closed
  in $\cM$. Then the interior of $S$ is empty.
\end{proposition}
\begin{proof}
  In this proof, we will assume that we have enlarged the set
  $\cM$ and extended the function $\BG$ in the obvious way so that the
  removable singularities are points contained in $\cM$ and $\BG$ is
  also defined at these points. Let $U$
  denote a connected component of the interior of $S$. By
  Propositions~\ref{P.parends} and~\ref{P.hypends}, the disks $D_i$ or
  the deleted points $p_i$ that are not removable singularities behave
  as local minima of the function $\BG$. Since $\De_{g_0}\BG=0$ both
  in $U$ and in a neighborhood of any removable singularity $p_i$, the
  maximum principle for harmonic functions ensures that the maximum of
  $\BG$ must be attained at a point $z$ of $\pd U$ (possibly a removable singularity).

  Since the boundary of $U$ is invariant, $z$ must be a critical point
  of $\BG$, which is necessarily isolated by
  Proposition~\ref{L:halfb}. As $\BG$ is smooth in a neighborhood of
  $z$, $z$ is an isolated maximum of $\BG$ in the closure $\overline U$ and
  $\BG$ is increasing along the local flow $\psi_t$ of
  $\nabla_{g_0}\BG$, it follows that for any $\ep>0$ there exists some
  $\de>0$ such that
\[
\psi_t\big( B_{g_0}(z,\de)\cap \overline U\big)\subset B_{g_0}(z,\ep)
\]
for all $t>0$ and
\[
\bigcap_{t\geq0} \psi_t\big( B_{g_0}(z,\de)\cap \overline U\big)=\big\{z\big\}\,.
\]
Hence there exists a region $B_{g_0}(z,\de)\cap \overline U$ of
nonzero measure whose $\om$-limit is $z$. This contradicts the fact
that, since $\De_{g_0}\BG=0$ in a neighborhood of~$z$, the local flow
of $\nabla_{g_0}\BG$ is area-preserving, so the set $U$ must be empty.
\end{proof}

Hence we immediately obtain the desired statement about $\cF$:

\begin{corollary}
The complement of $D$ in $\cM$ coincides with the boundary of $D$
in $\cM$, that is, 
\[
\cF=\cM\minus D=\pd D\,.
\]
\end{corollary}

Our goal now is to derive further properties of the set $\cF$.  For
technical reasons, it is more convenient to consider the flow of an
auxiliary vector field $X$ defined in the whole compact surface $\Si$
rather than that of $\nabla_{g_0}\BG$, which is only defined on $\cM$. For this, let us take a point
$q_i$ belonging to the interior of each disk $D_i$. Since the disk $D_i$ retracts into
$q_i$, one can take a diffeomorphism
\[
\Psi:\cM\to \Si\minus\{p_1,\dots,p_{\la_1},q_1,\dots,q_{\la_2}\}
\]
which is equal to the identity outside a small neighborhood of the closed
disks ${D_i}$ (in particular, at $\By$). 

Let us relabel the parabolic
ends if necessary so that $\{p_i\}_{i=1}^{\la_1'}$ are the
non-removable singularities, with $0\leq \la_1'\leq\la_1$. The auxiliary field is then defined as
\[
X:=F\,\Psi_*\big(\nabla_{g_0}\BG\big)\,,
\]
where $F:\Si\to\RR$ is a smooth nonnegative function that only
vanishes at $\{\By\}\cup \cP$, with
\begin{equation}\label{points}
\cP:=\big\{p_1,\dots, p_{\la_1'},q_1,\dots, q_{\la_2}\big\}
\end{equation}
and is chosen so that $X$ can be smoothly extended to the whole
surface $\Si$. It is standard that such a factor always exists;
notice, moreover, that we do not need to impose any additional
conditions on this factor to ensure that $X$ is also defined at the
removable singularities (we only need to consider the obvious
extension of $\BG$). Throughout this paper, we will denote the
flow of $X$ by $\phi_t$.

It is clear that
the function $\BG\circ \Psi^{-1}$ can then be extended to a continuous
function $\hG:\Si\to[-\infty,+\infty]$ by setting
\[
\hG(\By):=+\infty\,,\qquad \hG(p_i):=-\infty\,,\qquad \hG(q_j):=0
\]
for $1\leq i\leq \la_1'$ and $1\leq j\leq \la_2$. Hence it stems that
the field $X$ is gradient-like: the Lie derivative $\cL_X\hG$ is
strictly positive but at the zeros of the field $X$, which are the
only points at which the function $\hG$ can fail to be smooth. Moreover, the
zeros of $X$ are exactly the images under the diffeomorphism $\Psi$ of
the zeros of $\nabla_{g_0}\BG$ (including those that correspond to
removable singularities) and $\{y\}\cup\cP$. Notice that the
points in $\cP$ are precisely the global isolated minima of $\hG$: indeed,
by Proposition~\ref{P.hypends} if $\la_2\geq 1$, then $\BG$ is
positive and $\la_1'=0$. Besides, from this proposition it stems that
the cardinality of $\cP$ is
\begin{equation}\label{lap}
\la':=\begin{cases}
\la_1'&\text{if } \la_2=0,\\
\la_2 &\text{if } \la_2\geq1.
\end{cases}
\end{equation}

The reason why we consider the field $X$ is that, instead of directly
analyzing the sets $D$ and $\cF$ associated with the field $\nabla_{g_0}\BG$, it is easier
to consider the analogous dynamical objects for the field $X$. That
is, we will consider the set $\hD$ of the points in $\Si$  whose $\om$-limit along the flow of $X$ is
$\By$, which we will still call the {\em basin of attraction}\/ of the
field $X$. Its associated {\em basin boundary}\/  is then defined as
\[
\hF:=\Si\minus\hD\,.
\]
Just as in the case of the set~$D$, it is standard that $\hD$ is diffeomorphic to a
disk. 

\begin{remark}\label{R.X}
Proposition~\ref{P.interior} and its proof are also valid, mutatis
mutandis, for the field $X$. That is, if $S\subset \Si$ is a closed
invariant set under the flow of $X$ that does not contain the point
$\By$, it has empty interior. In particular, 
\[
\hF=\Si\minus\hD=\pd\hD\,.
\]
\end{remark}

The following proposition completely characterizes the set $\hF$ in
terms of the zeros of the field $X$ and their stable components. To
state this result, let us introduce the notation
\begin{equation}\label{cZ}
\cC:=\Psi\big[\big\{ z\in\cM\cup\{p_{\la_1'+1},\dots,p_{\la_1}\}: \nabla_{g_0}
\BG(z)=0\big\}\big]
\end{equation}
for the image under the diffeomorphism $\Psi$ of the critical points
of $\BG$, including those that may correspond to a removable singularity.

\begin{proposition}\label{P.hF}
$\hF$ is the union of
  the zeros of the field $X$ other than $\By$ and the stable sets of
  the zeros in the set $\cC$:
\[
\hF=\cP\cup \bigcup_{z\in\cC} W^s(z)\,.
\]
Furthermore, $\hF$ is connected. 
\end{proposition}

\begin{proof}
Remark~\ref{R.X} readily implies that $\hF$ is connected and has empty
interior. Let us now set
\[
W:=\cP\cup \bigcup_{z\in\cC} W^s(z)\,.
\]
By definition, it is clear that the $\om$-limit of any point $x\in W$
cannot be the point $\By$, so $W\subset \Si\minus\hD$.

Let us now prove the converse implication: $\Si\minus \hD\subset
W$. For this, we shall show that the
$\om$-limit of any $x\in \Si\minus \hD$ must be a zero of the field
$X$ different from $\By$. We can obviously suppose that $x$ is not a zero of
$X$. Since $X$ is gradient-like, $\phi_tx$ must tend to a zero $z$ of
$X$ as $t\to\infty$~\cite{MP82}. Besides, $z$ must belong to the set $\cC$ because the points in
$\cP$ are minima of $\hG$ and the Lie derivative
\[
\cL_X\hG(\phi_tx)=\frac d{dt}\hG(\phi_tx)
\]
is positive if $x$ is not a zero of $X$. Hence we infer that $\hF=W$.
\end{proof}

It is clear that a good description of $\hF$
immediately yields a complete characterization of the original set
$\cF$. In particular, Proposition~\ref{P.hF} easily implies that,
roughly speaking, $\cF$
consists of the set $\hF$ and some additional segments  that connect a
point in $\hF$ with a parabolic end of the surface. From this set, of
course, we still have to remove the points in $\hF$ corresponding to
an end. This is the content of the following

\begin{corollary}\label{C.cF}
The set $\cF$ is the image under the diffeomorphism $\Psi^{-1}$ of the set
\[
\bigg(\hF\cup
\bigcup_{i=\la_1'+1}^{\la_1}\big\{\phi_{-t}p_i:t>0\big\}\bigg)\Big\minus
\big\{p_1,\dots, p_{\la_1},q_1,\dots, q_{\la_2}\}\,.
\]
The closure of $\cF$ in $\Si$ is connected.
\end{corollary}

\begin{proof}
It is clear that the image under $\Psi^{-1}$ of the set $\hF$ minus
the ends 
\[
\big\{p_1,\dots, p_{\la_1}, q_1,\dots, q_{\la_2}\big\}
\]
must be
contained in $\cF$. Likewise, the image under $\Psi^{-1}$ of any point $x$
in $\hD$ will be
in the basin of attraction $D$ unless at some positive time $t$ the trajectory $\phi_tx$ passes
through an end, which is necessarily a removable singularity since the
other ends are all zeros of the field $X$. That is, for some $t>0$
one would have
\[
x=\phi_{-t}p_i\qquad\text{for some }{\la_1'}+1\leq i \leq
\la_1\,.
\]
This proves the formula in the statement. (Of course, among the
removable singularities it would be enough to consider those that are
not a zero of the field~$X$.)

Therefore, the closure of $\cF$ in $\Si$ is either empty (then $\hF$
consists of a single point) or is diffeomorphic to the union of $\hF$
and a finite number of segments whose endpoints are a removable
singularity $p_i$ as above and a zero of the field $X$ (which obviously
belongs to $\hF$). Therefore, the closure of $\cF$ is connected. 
\end{proof}

\section{Structure of the basin boundary and bounds for the critical points}
\label{S:bounds}

In this section we will provide a full characterization of the basin
boundary, thereby obtaining an upper bound for the number of
critical points of the Green's function. The characterization of the
basin boundary lay bare a strong connection between the dynamics of
the field $X$ (or, equivalently, $\nabla_gG$) and the topology and
conformal structure of the surface. 

The following theorem provides an upper bound for the number of
critical points of $\BG$ (including those corresponding to removable
singularities) in terms of the conformal properties of the
surface, which appear through the number $\la'$ introduced in
Eq.~\eqref{lap}. A straightforward consequence of this result is the
purely topological bound presented in Theorem~\ref{T:main}, which
differs from the present statement in that here we are using additional information on
the conformal structure of the surface $\Si$ to sharpen the upper bound:

\begin{theorem}\label{T.main2} 
The number of zeros of the field $\nabla_{g_0}\BG$ in 
\[
\cM\cup\big\{p_{\la_1'+1},\dots,p_{\la_1}\big\}\,,
\]
that is, the cardinal of the set $\cC$, is not
larger than $2\nu+\la'-1$, and if this upper bound is attained then $\BG$ is Morse.  
\end{theorem}

Before presenting the proof of this result, it is illustrative to
sketch the argument in the easiest case: when $M$ is diffeomorphic to
$\RR^2$. In this case, the proof is similar to that of a result on the
absence of critical points in some boundary value problems in the
exterior of a bounded domain in~$\RR^n$
that we proved in~\cite{EP08}. Of course, this particular case is
elementary and could be easily treated using the uniformization theorem, but it serves to
illustrate the basic dynamical ideas underlying the proof of the
general situation.

\begin{proof}[Sketch of the proof when $M$ is diffeomorphic to $\RR^2$]
  Let us analyze what happens in $\cM$, which is diffeomorphic to the sphere minus a
  point $p$. Suppose that we have a zero $z$. By
  Proposition~\ref{L:halfb}, its stable set (with
  the point $z$ deleted) consists of at least two curves. The $\al$-limit
  of each of these curves cannot be $\By$ by Proposition~\ref{L.pole},
  so as the vector field $X$ is gradient-like either the curve approaches
  the end $p$ or its $\al$-limit is another zero $z_1$. We can now
  apply the same argument for the zero $z_1$, and if necessary to
  successive zeros $z_2,z_3,\dots$  Notice that $z_j\neq z_k$ for
  $j\neq k$, since otherwise we could have an invariant set (defined
  by the union of curves in the stable sets of zeros of the field)
  that is a Jordan curve not containing $\By$. Then this invariant set would separate the
  plane in two disjoint invariant sets with nonempty interior,
  contradicting Proposition~\ref{P.interior}.

  Since the zeros are isolated in $\cM$ by Proposition~\ref{L:halfb}, we
  can eventually take a union of curves in the stable sets of zeros
  $z_j$ of the field (possibly infinitely many, but only accumulating
  at $p$) whose closure in the sphere is a Jordan curve that
  contains the point $p$ but not $\By$. Again, this curve encloses an
  invariant set with nonempty interior that does not contain the point
  $\By$, in contradiction with Proposition~\ref{P.interior}. Hence
  there cannot be any zeros of $\nabla_{g_0}\BG$ and the theorem follows.
\end{proof}

In the general case, the proof is more involved and relies on a careful
analysis of the saddle connections between zeros of the field. In the
demonstration we need two lemmas that are presented right after
the proof and make use of the same notation. 

The proof is divided in two parts. First we show that the number of
zeros of $X$ is finite. Notice this is not trivial, since the zeros of
$X$ could accumulate at the set $\cP$, which is associated with
ends of the surface $M$, without contradicting the fact that critical
points of $\BG$ are isolated by Proposition~\ref{L:halfb}. However, we
show that if there were an infinite number of critical points, there
would be infinitely many closed invariant curves in $\Si$ defining
independent homology classes, which is forbidden by the fact that the
fundamental group of $\Si$ has finite rank. Roughly speaking, these closed invariant
curves are constructed by successive continuation of the stable sets of some
zeros of $X$. The second part of the proof consists in estimating the
number of zeros of $X$ using Hopf's index theorem and the
characterization of the dynamics of $X$ in a neighborhood of each zero.

\begin{proof}[Proof of Theorem~\ref{T.main2}]
Let us start by recalling that the set $\cP$, introduced in
Eq.~\eqref{points}, consists of precisely $\la'$ points, which are the
minima of the function $\hG$. As
$X$ is gradient-like, the $\al$- and $\om$-limit sets of any
trajectory of this field are necessarily a zero of $X$. Moreover, the
zeros of this field are obviously given by the set
\[
\cZ:=\{\By\}\cup \cC\cup \cP\,,
\]
where we record here that the set $\cC$, defined in~\eqref{cZ},
corresponds to the critical points of $\BG$ (possibly including
removable singularities) under the diffeomorphism $\Psi$.  The fact
that the critical points of $\BG$ are isolated (by
Proposition~\ref{L:halfb}) guarantees that $\cZ$ does not accumulate
but possibly at $\cP$.

Let $\gamma$ be a trajectory of the field $X$. This trajectory will be
called {\em constant} if it consists of a single point. We will use the notation $\al(\gamma)$
and $\om(\gamma)$ for the $\al$- and $\om$-limit sets of $\gamma$. 
Let us introduce a partial order on the set of zeros $\cZ$ as follows. Given two
points $x,x'\in \cZ$, we shall write $x\succ x'$ if, for any
open neighborhoods $U\ni x$ and $V\ni x'$ in $\Sigma$, there exist integers
$p\leq0$, $q\geq0$ and nonconstant trajectories (of the field $X$)
$\gamma_{p},\dots,\gamma_{q}$ such that
\begin{enumerate}
\item $\om(\gamma_{p})\in U$, $\al(\gamma_{q})\in V$,
\item $\al(\gamma_j)=\om(\gamma_{j+1})$ for $p\leq j\leq q-1$.
\end{enumerate}

\subsubsection*{Finiteness}
We claim that the set of zeros $\cZ$ is finite. In order to
prove this, let us assume the contrary. By Lemma~\ref{L:h} below, for each point
$x\in\cC$ there exists some point $p\in\cP$ such that
$x\succ p$. Since $\cP$ is finite, there exists some $p_1\in \cP$ such that one
can choose a sequence $(x_k)_{k=1}^\infty$ of distinct points in
$\cC$ with $x_k\succ p$. For each point $x_k$,
Lemma~\ref{L:curve} below yields a continuous path $\Gamma_{k,1}:[0,1]\to\Si$ whose image is
invariant under the field $X$ and satisfies $\Gamma_{k,1}(0)=x_k$ and
$\Gamma_{k,1}(1)=p_1$.

As a straightforward consequence of Proposition~\ref{L:halfb}, the invariant set
\[
W^s(x_k)\minus\Gamma_{k,1}([0,1])
\]
is nonempty. If we
let $\tilde x_k$ be the $\al$-limit of a trajectory contained in
this set, we obviously have $x_k\succ\tilde x_k$. By Lemma~\ref{L:h}, either $\tilde
x_k\in\cP$ or $\tilde x_k\succ p$ for some $p\in\cP$. Since $\cP$ is
finite, by Lemma~\ref{L:curve} and possibly upon restricting
ourselves to a subsequence that we still denote by $(x_k)_{k=1}^\infty$, we obtain a family of continuous
paths $\Gamma_{k,2}:[0,1]\to\Si$ whose image is invariant under $X$ and such
that $\Gamma_{k,2}(0)=x_k$ and $\Gamma_{k,2}(1)=p_2$ for some fixed $p_2\in\cP$ (possibly
the same as  $p_1$).

By construction, for each positive integer $k$ the connected set
\[
\Gamma_{k,1}([0,1])\cup \Gamma_{k,2}([0,1]) \cup
\Gamma_{k+1,1}([0,1]) \cup \Gamma_{k+1,2}([0,1])
\]
consists of a continuous curve that connects the points $x_k$ and
$x_{k+1}$ passing through $p_1$ and another continuous curve that
connects the same pair of points $x_k,x_{k+1}$ passing through
$p_2$. It is then evident that this set, which can have self-intersections, contains an invariant loop
(continuous closed curve) $\La_k$. One can obviously ensure that
$\La_k\neq\La_{k'}$ for $k\neq k'$ (these loops can intersect,
though), and that the point $\By$ does not belong to any $\La_k$.

For any integer $j$, the union of invariant loops 
\[
\bigcup_{k=1}^j\La_k
\]
cannot disconnect $\Si$,
since a connected component of $\Si\minus \bigcup_{k=1}^j\La_k$ that does not contain
$\By$ would be an invariant set with nonempty interior, contradicting
Proposition~\ref{P.interior}. Therefore, it is standard that the homology classes
$[\La_k]\in H_1(\Si;\ZZ)$ defined by the cycle $\La_k$ must be
independent for all $k=1,2,\dots$ This is impossible in a finitely
generated surface, so we infer that the set $\cC$ is finite.

\subsubsection*{Upper bound}
Let us now pass to bound the cardinality of the set $\cC$. Suppose
that the number $\la_2$ of removed disks is at least one, so that
\[
\cP=\big\{q_1,\dots, q_{\la_2}\}\,.
\]
A first observation is that the points $q_i$ are isolated zeros of the
field $X$, by the finiteness of $\cZ$, and are local repellers because
they correspond to minima of the function $\hG$. Therefore, the index
of $X$ at these points is
\[
\ind(q_i)=1\,.
\]
Similarly, the point $\By$ is an isolated zero which is a local
attractor, so it has index $1$. If we now apply Hopf's index theorem
to the vector field $X$ in $\Si$, we get that the sum of the indices of
the zeros of $X$ equals the Euler characteristic of
$\Si$:
\begin{equation}\label{eq1}
\ind(\By)+\sum_{z\in\cC} \ind(z)+\sum_{i=1}^{\la_2}\ind(q_i)=\chi(\Si)=2-2\nu\,.
\end{equation}
Since the
index of each point $z\in\cC$ is smaller than or equal to $-1$ by
Proposition~\ref{L:halfb}, plugging the values of the indices of $\By$
and $q_i$ we
find
\begin{equation}\label{nuevaeq}
\# \cC\leq -\sum_{z\in\cC} \ind(z)=2\nu+\la_2-1\,.
\end{equation}
The equality is not
satisfied but perhaps when $\ind(z)=-1$ for all $z\in\cC$, that is, when $\BG$ is
Morse (by Proposition~\ref{L:halfb}). This proves the theorem when $\la_2\geq1$.

Consider now the case where $\la_2=0$, so that 
\[
\cP=\big\{p_1,\dots,p_{\la_1'}\big\}\,.
\]
Arguing as before one easily finds that the index of $X$ at each point
$p_i$ is
\[
\ind(p_i)=1\,.
\]
We now apply Hopf's index theorem to the vector field
$X$ in $\Si$ to find
\begin{equation}\label{eq2}
\ind(\By)+\sum_{z\in\cC} \ind(z)+\sum_{i=1}^{\la_1'}\ind(p_i)=\chi(\Si)=2-2\nu\,,
\end{equation}
so that the same argument as above yields
\begin{equation}\label{nuevaeq2}
\#\cC\leq-\sum_{z\in\cP} \ind(z)=2\nu+\la_1'-1\,.
\end{equation}
Again, the inequality being saturated at most
when $\ind(z)=-1$ for all $z\in\cC$ (that is, when $\BG$ is
Morse). The theorem then follows.
\end{proof}

\begin{lemma}\label{L:curve}
Let $x,x'\in\cZ$ such that $x\succ x'$. Then there
exists a (not necessarily unique) injective continuous path
$\Gamma:[0,1]\to\Si$ such that:
\begin{enumerate}
\item $\Gamma(0)=x$ and $\Gamma(1)=x'$.

\item $\hG\circ\Gamma$ is strictly decreasing.

\item The curve $\Gamma([0,1])$ is invariant under the field $X$.
\end{enumerate}
\end{lemma}
\begin{proof}
By the definition of the partial order $\succ$ and Zorn's lemma, there exists a countable
sequence $\{\gamma_j\}_{j=\overline p}^{\overline q}$ ($-\overline p,\overline q\in\NN\cup\{\infty\}$) of
nonconstant trajectories of the field $X$ satisfying Conditions (i) and (ii) in the proof of Theorem~\ref{T.main2}, and
such that
\[
\lim_{j\to\overline  p}\om(\ga_j)=x\,,\qquad \lim_{j\to \overline q}\al(\ga_j)=x'\,.
\]
Let us consider any continuous parametrization
\[
\Gamma:[0,1]\to\overline{\bigcup_{j=\overline p}^{\overline q}\gamma_j(\RR)}\subset\Si
\]
mapping $0$ to $x$ and $1$ to $x'$. Since
the Lie derivative $\cL_X\hG$ is positive in $\Si\minus\cZ$, $\hG$ is 
increasing along nonconstant trajectories, which implies that $\hG\circ\Gamma$ is
strictly decreasing (notice that the definition of $\Ga$ accounts for
the fact that this function is decreasing instead of
increasing). Therefore, $\Ga$ is injective. Moreover, the curve $\Gamma([0,1])$ is clearly invariant
because it is the union of trajectories.
\end{proof}

\begin{lemma}\label{L:h}
For any $z\in\cC\cup\{\By\}$ there exists
some $p\in\cP$ such that $z\succ p$. Moreover, there are no $x\in\cZ$
such that $x\succ x$ or $p\succ x$ for some $p\in\cP$ (i.e., the elements in $\cP$
are minimal with respect to the partial order).
\end{lemma}
\begin{proof}
Let us take an element $z\in\cC$. By
Proposition~\ref{L:halfb}, there exists a neighborhood $U$ of
$z$ such that $(W^s(z)\cap U)\minus\{z\}$ has at
least two components $C_1,C_2$ and each $C_i$ is a piece of a
trajectory of $X$. Since $X$ is gradient-like, the $\al$-limit set
of $C_1$ is another zero $z_1\in\cC\cup\cP$.

If $z_1\in\cP$, the statement follows. Otherwise, we can repeat the
previous argument replacing $z$ by $z_1$. As $X$ is gradient-like,
proceeding this way we obtain a sequence of distinct points
$z_1,z_2,\dots$ in $\cZ$. If there is some point $z_k\in\cP$, we are
done, so we can assume that there is a sequence $(z_k)_{k=1}^\infty$
of distinct points of $\cZ$ with $z_k\succ z_{k+1}$. Since
$\Psi^{-1}(\cC)$ consists of isolated points in $\cM$ by
Proposition~\ref{L:halfb}, it follows that $\dist_{g_0}(z_k,\cP)$
tends to zero as $k\to\infty$.  Hence there must exist some $p\in\cP$
such that a subsequence $(z_{k'})_{k'\in J}$
tends to $p$. As above, it then follows that the initial sequence
$(z_k)_{k=1}^\infty$ also tends to $p$ because of the fact that $p$ is an isolated minimum of $\hG$ and
$\hG$ is increasing along trajectories.

Notice that the proof also applies when we start with the point $\By$
instead of a point $z\in\cC$. Finally, note that obviously $x\not\succ x$ for any $x\in\cZ$ by
Lemma~\ref{L:curve}, and that $p\not\succ x$ because each $p\in\cP$ is an
isolated minimum of $\hG$ and $\hG$ is increasing along the flow of the
field $X$.
\end{proof}

In the following theorem we show that the basin boundary $\hF$ encodes the
topology of the surface $\Si$. In particular, the flow of the field $X$
defines in a natural way a decomposition of $\Si$ into a disk $\hD$
and its $1$-skeleton $\hF$. This decomposition is similar to the one
arising when one considers the cut locus of a point in $\Si$ (with respect
to some metric $g_0$) but it is generally different. Together with
Proposition~\ref{P.hF} and Corollary~\ref{C.cF}, this provides a rigorous reformulation of the
Heuristic Principle stated in the Introduction.

\begin{theorem}\label{P.hF2}
The set $\hF$ is a connected graph with the same homology as the
surface $\Si$:
\[
H_1(\hF;\ZZ)=\ZZ^{2\nu}\,.
\]
\end{theorem}
\begin{proof}
We showed in Proposition~\ref{P.hF} that $\hF$ consists of the zeros
of $X$ other than~$\By$ (that is, $\cC\cup\cP$) and their stable
components, which are continuous curves with endpoints belonging to
$\cC\cup\cP$. Since $\cC$ is finite by Theorem~\ref{T.main2}, it then
follows that $\hF$ is a connected graph.

Let $B$ be a small disk in $\Si$ centered at $\By$ and let us consider
its image under the time-$t$ flow of $X$, $\phi_t(B)$. It is apparent
that $\Si\minus\{\By\}$ deform retracts onto the set 
\[
\Si\minus\phi_t(B)\,,
\]
for any $t\leq 0$. Moreover, 
\[
\hF=\bigcap_{t\leq0}\big(\Si\minus\phi_t(B)\big)
\]
by the definition of the set $\hF$. This ensures that $\hF$ is a
strong deformation retract of $\Si\minus\{\By\}$, so it is well known
that, $\hF$ being a connected graph, we have the isomorphism of
homology groups
\[
H_1(\hF;\ZZ)=H_1(\Si\minus\{\By\};\ZZ)\,.
\]
Since 
\[
H_1(\Si\minus\{\By\};\ZZ)=H_1(\Si\minus B;\ZZ)=H_1(\Si;\ZZ)=\ZZ^{2\nu}
\]
by a standard argument using the Mayer--Vietoris sequence, the theorem follows.
\end{proof}

\begin{remark}
Because of the connection between $\cF$ and $\hF$, Theorem~\ref{P.hF2}
also gives very detailed information about the structure
of the set $\cF$. In particular, $\cF$ is a graph but is necessarily noncompact
and possibly disconnected. Moreover, the rank of $H_1(\cF;\ZZ)$ is at most
$2\nu$ but can be strictly smaller than this number, as some of the
cycles that appear in $\hF$ can be killed after removing the points of
$\Si$ that correspond to the ends of the noncompact surface.
\end{remark}

To conclude, let us present an illustrative example in which the different sets
that we have been discussing in Sections~\ref{S:basin}
and~\ref{S:bounds} can be computed explicitly:

\begin{example}
Let $M$ be the torus minus the point $p:=(0,\pi)$, written in terms
of the standard $2\pi$-periodic coordinates on the torus. We choose a
complete, conformally flat metric $g$ on $M$, and fix the position of
the pole at $y=(0,0)$. It is then standard that Li and Tam's
procedure~\cite{LT87} gives rise to a Green's function $G$ invariant under
the isometric transformations of the flat torus that fix the point
$y$.

It is then clear that the curves
\[
\{0\}\times\SS^1\,,\quad \{\pi\}\times\SS^1\,,\quad \SS^1\times
\{0\}\,, \quad \SS^1\times \{\pi\}
\]
are then invariant under the local flow of $\nabla_gG$. An easy
argument using the symmetries then show that the points $(\pi,0)$ and
$(\pi,\pi)$ are zeros of the field $\nabla_gG$, which must be
nondegenerate (hence hyperbolic saddles). These are the only zeros because
the upper bound in Theorem~\ref{T.main2} is attained.

The set $\hF$
consists of the two circles $\{\pi\}\times\SS^1$ and
$\SS^1\times\{\pi\}$, which respectively correspond to (two) saddle
connections and to the trajectories of $X$ that connect a saddle with
the point $p$. Since $\hF$ contains two independent cycles,
$H_1(\hF;\ZZ)$ is isomorphic to the first homology group of the
torus. The set $\cF$ is then given by
\[
\cF=\hF\,\minus\{p\}\,,
\]
so is consists of a closed curve and an open curve. In particular, $H_1(\cF;\ZZ)=\ZZ$.
\end{example}

\section{Applications and remarks}
\label{S:remarks}

In this last section we will make some remarks about the connections
between the critical points of the Green's function and the
conformal properties of the surface. We will end with some comments about
Green's functions on surfaces of infinite topological type.

\subsection*{Conformal structure and dynamics} A consequence of
Theorem~\ref{T.main2} is that by analyzing the dynamics of the gradient field
$\nabla_gG$ (or, equivalently, of the field~$X$) we can sometimes
extract information about the conformal structure of the underlying
surface. For example, a surfaces must satisfy very stringent geometric
conditions in order to admit a Green's function without any critical
points, as we show in the following 

\begin{proposition}
Let us suppose that the surface $M$ has at least two ends. If all its
ends are either hyperbolic or non-removable singularities, then any
Green's function on $M$ has at least one critical point.
\end{proposition} 
\begin{proof}
Eqs.~\eqref{nuevaeq} and~\eqref{nuevaeq2} mean that, with the same notation as
in the proof of Theorem~\ref{T.main2},
\begin{equation}\label{eq3}
-\sum_{z\in\cC}\ind(z)=2\nu-1+\la'\,.
\end{equation}
Let us call
\[
\cC_1:=\cC\minus\big\{ p_{\la_1'+1},\dots, p_{\la_1}\big\}
\]
and $\cC_2:=\cC\minus\cC_1$. Any of these sets can be empty, and it is
clear that the cardinality of $\cC_1$ equals the number of critical
points of the Green's function $G$ in $M$. The cardinality of $\cC_2$
is at most $\la_1-\la_1'$.

As $\la'=\la_1'+\la_2$, Eq.~\eqref{eq3} can be rewritten as
\begin{equation}\label{laultima}
-\sum_{z\in\cC_1}\ind(z)=2\nu-1+\la_1'+\la_2+\sum_{z\in\cC_2}\ind(z)\,.
\end{equation}
If all the ends are hyperbolic, $\la_1=\la_1'=0$, and if all ends are
non-removable singularities, then $\la_1=\la_1'$ and
$\la_2=0$. Therefore, in both cases $\cC_2$ is empty, so the RHS
of~\eqref{laultima} is nonzero and $\cC_1$ cannot be empty. Since
hyperbolic ends and non-removable singularities cannot coexist, the proposition is proved.
\end{proof}

\subsection*{Surfaces of infinite topology} 

It is worth emphasizing that the case of surfaces whose fundamental
group is not finitely generated is totally different from the case of
surfaces of finite type. We shall next provide examples of surfaces
that are not finitely generated both with an infinite number of
critical points and without any critical points.

\begin{example}
Let us consider the case where
$M$ is a torus of infinite genus with two reflection symmetries (see Figure~\ref{F.torus}). For
this, we can regard $M$ as a surface embedded in $\RR^3$ and invariant
under the reflections 
\[
\Pi_1(x_1,x_2,x_3):= (-x_1,x_2,x_3) \quad\text{and}\quad \Pi_3(x_1,x_2,x_3):= (x_1,x_2,-x_3)\,.
\]
We will endow $M$ with the metric $g$ induced by the Euclidean metric in
$\RR^3$ and consider a Green's function $G$ on $M$ (satisfying
Assumptions~1 and~2 in Section~\ref{S:LiTam}) with a pole at a point
$y$ invariant under the reflections $\Pi_1,\Pi_3$. Li and Tam's
procedure can be used to obtain a Green's function with the same
symmetries (that is 
\[
G=G\circ \Pi_j
\]
for $j=1,3$).

\begin{figure}[t]
\includegraphics[scale=0.25]{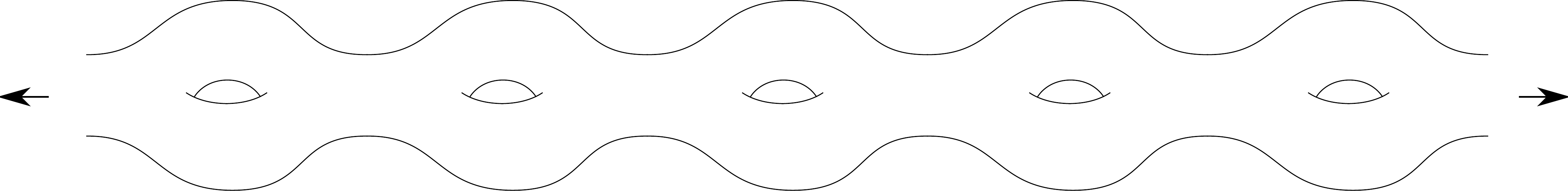}
\caption{A torus of infinite genus with two reflection symmetries.} 
\label{F.torus}
\end{figure}

It is straightforward that the intersection of the surface $M$ with
the plane $\{x_j=0\}$ is invariant under the local flow of $\nabla_g
G$, for $j=1,3$. A simple argument then shows that the intersection of
$M$ with the $x_2$-axis must be invariant too. Since it consists of
isolated points, these must therefore be either critical points of $G$
or the pole $y$. These
critical points are obviously infinite in number.
\end{example}

\begin{example}
Let us consider the unit disk
\[
\DD^2:=\big\{(x_1,x_2): x_1^2+x_2^2<1\big\}
\]
with its hyperbolic metric $g_1$. We let $M$ be the unit disk with 
infinitely points removed as
\[
M=\DD^2\minus\bigg((0,0)\cup\bigcup_{n=2}^\infty\Big(0,\frac1n\Big)\bigg)\,.
\]
The hyperbolic metric $g_1$ is not complete on $M$, but it is well
known~\cite{Nomizu} that there is a conformally equivalent metric
$g=\chi g_1$ such that $(M,g)$ is complete. 

Consider the minimal Green's function $G$ with a pole at a point $y\in
M$. By the conformal invariance of the Laplacian, this Green's
function is precisely the minimal Green's function of the disk $\DD^2$
with the hyperbolic metric $g_1$ (or rather its restriction to
$M$). Since the gradient of the minimal Green's function of the
hyperbolic disk $(\DD^2,g_1)$ does not vanish by
Theorem~\ref{T.main2}, $(M,g)$ is an example of a surface that is not
finitely generated whose minimal Green's function does not have any
critical points. Incidentally, notice that, although the proof of
Theorem~\ref{T.main2} does not apply to surfaces of infinite
topological type, zero is exactly the upper bound one gets by
recklessly applying the formula $2\nu-1+\la_2$ in this case (notice
that for $M$ we have $\nu=0$, $\la_1=\infty$, $\la_1'=0$, $\la_2=1$).
\end{example}

\section*{Acknowledgments}

This work is supported in part by the Spanish MINECO under
grants~FIS2011-22566 (A.E.), MTM2010-21186-C02-01 (D.P.-S.) and the ICMAT Severo
Ochoa grant SEV-2011-0087 (A.E.\ and D.P.S.), and by Banco
Santander--UCM under grant~GR35/10-A-910556 (A.E.). The authors
acknowledge the Spanish Ministry of Science for financial support
through the Ram\'on y Cajal program.

\bibliographystyle{amsplain}

\end{document}